\documentclass[12pt,draft]{amsart}

\usepackage{a4}
\usepackage{amscd,amssymb,booktabs,scalefnt}
\usepackage[utf8]{inputenc}
\usepackage[T1]{fontenc}

\newtheorem{theorem}{Theorem}[section]
\newtheorem{prop}[theorem]{Proposition}
\newtheorem{lemma}[theorem]{Lemma}
\newtheorem{cor}[theorem]{Corollary}
\newtheorem{conj}[theorem]{Conjecture}
\newtheorem*{conjstar}{Conjecture}

\theoremstyle{definition}
\newtheorem{defn}[theorem]{Definition}

\newtheorem{example}[theorem]{Example}

\def\Bbb{\mathbb} \def\cal{\mathcal}
\def\wt#1{\widetilde{#1}}
\def\wl#1{\overline{#1}}
\def\sier#1{{\cal O}_{#1}}
\def\C{{\Bbb C}}  \def\P{{\Bbb P}} \def\R{{\Bbb R}}
 \def\cX{{\cal X}}
\def\cA{{\cal A}} \def\cR{{\cal R}} \def\cK{{\cal K}} \def\cG{{\cal G}}
\def\vp{\varphi}
\DeclareMathOperator{\Rk }{Rank}
\DeclareMathOperator{\moda}{Mod}

\newdimen\digitwidth
\def\n-{\rule{1\digitwidth}{0.4pt}}

\begin{document}

\title{\bf Simple curve singularities}
\author{Jan Stevens}
\address{Department of Mathematical Sciences, Chalmers University of 
Technology and University of Gothenburg. 
SE 412 96 Gothenburg, Sweden}
\email{stevens@chalmers.se}

\begin{abstract}
In this paper we classify simple parametrisations of complex
curve singularities. Simple means that 
all neighbouring singularities fall in finitely many equivalence classes. 
We take the neighbouring singularities to be the ones occurring 
in the versal deformation of the parametrisation. This leads to a smaller
list  than that obtained by looking at the neighbours in a fixed space
of multi-germs. Our simple parametrisations are the same as the 
fully simple singularities of Zhitomirskii, who classified real plane and
space curve singularities. The list of simple parametrisations of plane
curves is the A-D-E list. Also for space curves the list coincides
with the lists of simple curves of Giusti and
Fr\"uhbis-Kr\"uger, in the sense of deformations of the 
curve. For higher embedding dimension no classification of simple
curves is available, but we conjecture that even there the list is
exactly that of curves with simple parametrisations. 
\end{abstract}

\maketitle

\section*{Introduction}
Curve singularities can be described by parametrisations
or by systems of equations. These two view points lead to 
different list of simple objects, with simple meaning that all neighbouring
singularities fall in finitely many equivalence classes. 
This phenomenon was already observed by Bruce and Gaffney,
who classified simple parametrisations of irreducible plane
curve singularities \cite{BrGa}.
In this setting the neighbouring singularities are to be found among 
the maps $(\C,0)\to(\C^2,0)$,
with image given by an irreducible function, 
whereas in Arnold's A-D-E
classification \cite{ar} all functions are considered.
The classifications were extended to irreducible space curves by Gibson
and Hobbs \cite{GH},  irreducible curves of any embedding dimension
by Arnold \cite{ar2} and finally to reducible curves by 
Kolgushkin and Sadykov \cite{KS} on the one hand and to 
complete intersections by Giusti \cite{gi} and  determinantal
codimension 2 singularities by Fr\"uhbis-Kr\"uger \cite{FK, FN}
on the other hand.

A more restricted definition of simpleness for parametrisations 
was given by Zhitomirskii, who introduced fully simple singularities
\cite {zh}. The idea is that 
the neighbouring singularities of multi-germs of maps
should be all curves in
the neigbourhood of the image, even those with more irreducible
components. For plane curves he finds exactly the A-D-E 
singularities, and also his list of space curves (when corrected) 
coincides with the lists of Giusti and Fr\"uhbis-Kr\"uger together.
The definition is quite natural from the point of view of a somewhat
different approach to simpleness and modality, explicitly
formulated by Wall \cite {wall}. 
Given a singularity, the neighbouring singularities
are those occurring in its versal deformation.
For contact equivalence this yields the same concept of
simpleness as the one obtained by using the space of all functions. 
For a parametrisation $\vp \colon (\wl C,\wl 0) \to (\C^n,0)$,
where $ (\wl C,\wl 0)$ is a smooth multi-germ, we can consider
deformations of the map $\vp$ (see \cite[II.2.3]{GLS}, and \cite{GC}).
We call the parametrisation \textit{simple}, if there
are only finitely many isomorphism classes in  the versal deformation
of $\vp$. A curve is fully simple in the sense of Zhitomirskii \cite{zh}
if and only
if its parametrisation is simple.

In this paper we classify simple parametrisations of any
embedding dimension, for complex map germs. 
Rather than striking the non-simple ones from the
long lists of \cite{ar2, KS} we start from scratch; it is however a good
check to compare our list with theirs. Proving simpleness is more
difficult in our context, whereas showing that a singularity is not 
simple is easier: in all cases we succeed by giving a deformation
to a confining singularity. The list of these is very simple
and contains only the $L_{n+2}^n$, the curves consisting of
$n+2$ lines through  the origin in $\C^n$. For $n=1$ and $n=2$ the
definition has to be modified (Definition \ref{lnn2}).

For a plane curve singularity every deformation of the parametrisation
gives a deformation of the image curve, but not every deformation
of the curve comes from a deformation of the parametrisation:
a necessary and sufficient condition is that the $\delta$-invariant
is constant (see \cite[II.2.6]{GLS}). Without comparing lists we prove that a 
plane curve with simple parametrisation is itself simple by showing
that a deformation to a confining singularity always can be realised by a 
deformation of the parametrisation. We use the characterisation
of simple plane curves, given by Barth, Peters and Van de Ven,
as curves without points of multiplicity four on the (reduced)
total transform in each step of the embedded resolution \cite[II.8]{BPV}.

For space curves the $\delta$-invariant can go down in
a deformation of the parametrisation. Then it is not a (flat)
deformation of the image. The simplest example is that of two intersecting
lines which are moved from each other, forming two skew lines.
In this case we have only a partial explanation of the coincidence
of the two classifications. The simple parametrisation come in infinite
series, which all are deformations of $A_k\vee L_n^n$, $D_k\vee L_n^n$
or $E_k\vee L_n^n$ (the union of a plane germ with $n$ smooth
branches in independent directions), and a finite number of sporadic
parametrisations. The sporadic curves  have $\delta\leq 5$. As 
$\delta\geq 5$ for all confining singularities, all curves with
$\delta\leq 4 $ are simple, and non-simple curves with $\delta=5$ 
have a $\delta$-constant deformation to a confining singularity.

Beyond embedding dimension three not much is known about
simpleness of curves, in the sense of deforming the image.
The curves $L_r^r$, having $\delta =r-1$ are simple \cite[7.2.8]{BuGr},
and also the curves with $\delta=r$ \cite{gr}. This follows
because $\delta-r+1$ is upper semi-continuous. 
Determining adjacencies by explicit computations
with the versal deformation seems prohibiting difficult, as may 
be seen from our computations for partition curves \cite{st-u}.
Any parametrisation of a curve of multiplicity $m$ can be deformed to
a parametrisation of $L_m^m$, but this is not true for deformations
of the image. As shown by Mumford, there exist non-smoothable 
curves, who only deform
to curves of the same type, cf.~\cite{gr}. The argument is that the 
number of moduli is too large compared to the dimension of a smoothing component.  
Our lack of knowledge is shown by the 
old unsolved question whether  rigid reduced
curve singularities exist. Such a singularity, having no nontrivial
deformations at all, is certainly simple.
But we expect them not to exist. In fact, we believe that our list 
is also the list of simple curves (for the problem of deforming the
image).

\begin{conjstar}
The simple reduced curve singularities are exactly those with
simple parametrisation.
\end{conjstar}

The contents of this paper is as follows. After defining the basic
concepts and fixing our notations we formulate our main results. 
We give the list of simple parametrisations in Section 4. The proof
of the classification is in the next Section. In Section 6
we treat plane curves,  while the final Section discusses our 
Conjecture about simple curves. The Appendix contains
equations and parametrisations for the simple space curves, together
with the names from \cite{gi} and \cite{FK}.

\section{Basic concepts}
\subsection{Simple curves and parametrisations}
We consider germs of irreducible complex curves $(C,0)$, classified
up to analytic isomorphism. 
Let $n\colon  (\wl C,\wl 0) \to (C,0)$ be the normalisation.
Here $ (\wl C,\wl 0)$ denotes a smooth multi-germ.
The $\delta$-invariant of the curve is 
$\delta(C)=\dim_\C\sier {\wl C}/\sier C$.
Given an embedding $i\colon  (C,0)\to (\C^n,0)$ the composed 
map $\vp = i\circ n\colon  (\wl C,\wl 0) \to (\C^n,0)$
is a \textit{parametrisation} of the curve.
Classifying curves is equivalent to classifying parametrisations.

We can now consider two deformation problems, that of deforming
the curve, and that of deforming the parametrisation. These are very
different problems. By a result of Teissier a deformation of the
parametrisation gives a deformation of the curve and vice versa
if and only if the $\delta$-invariant is constant (see \cite[II.2.6]{GLS}).
In a deformation of the curve the number of components can go down:
a simple example is the deformation of $A_3$ into $A_2$, given 
by $y^2=x^4+s x^3$. In a deformation of the parametrisation the
number of components cannot decrease. The simplest example
of  deformation of the parametrisation which does not give a 
deformation of the image curve, is the deformation of $A_1\subset
\C^3$, which pulls apart the two lines. The first branch is 
parametrised by $(x,y,z)=(t_1,0,0)$, while the second is
$(x,y,z)=(0,t_2,s)$.
The ideal $I$ of the  image needs four generators:
\[
I=\bigl(yx,zx,y(z-s),z(z-s)\bigr)\;.
\]
For $s=0$ the ideal defines the two intersecting lines together with an
embedded component at the origin. 

Given a deformation problem, suppose that every object $X$
has a versal deformation $\cX\to S$.
\begin{defn}\label{defsimple}
An object $X$ is \textit{simple} if  there occur only finitely many
isomorphism classes in the versal deformation  $\cX\to S$.
\end{defn}
So an object is simple if it has no moduli and it also does not deform
to objects with moduli.
\begin{defn}
A collection of objects forms a collection of  \textit{confining}
objects, if no object of the collection is simple, and every 
other non-simple object deforms into one of the objects of the
collection.
\end{defn}

In particular, the two deformation problems for curve singularities
give two notions of simpleness. We will refer to the simple objects as
\textit{simple parametrisations}, and \textit{simple curves} respectively.

\subsection{$\cA$-simple map germs}
The first results on simple curve singularities were obtained
by Bruce--Gaffney \cite {BrGa}, for irreducible plane curve singularities,
using a different concept of simpleness obtained by considering 
parametrisations in a fixed space of multi-germs.
In fact for any of Mather's groups $\cR$, $\cK$ and $\cA$ (say $\cG$) 
one can define the notion of
a $\cG$-simple map germ $(k^n, 0)\to (k^p,0)$,  where $k$ is $\R$ or 
$\C$: a germ is \textit{$\cG$-simple}, if all neighbouring singularities
in the space of map germs  $(k^n, 0)\to (k^p,0)$  fall into finitely
many $\cG$-equivalence classes.

A parametrisation of an irreducible complex plane curve singularity
is a map 
germ $\vp\colon(\C, 0)\to (\C^2,0)$. Two such map germs $\vp_1$
and $\vp_2$ are $\cA$-equivalent if and only  defining equations
$f_1$ and $f_2$ for their images are $\cK$-equivalent, but 
$\cA$-simpleness of $\vp$ is not equivalent to 
$\cK$-simpleness of a defining equation $f$.
\begin{example}\label{AKexample}
The germ $\vp(t)=(t^4,t^5)$ is $\cA$-simple \cite[Theorem 3.8]{BrGa}
but its defining equation $f=y^4-x^5$ is unimodal; it 
is $W_{12}$ in Arnold's notation. The function $f$ has a
deformation $F(x,y,s)=y^4-x^5+s(x^2y^2+x^4)$ to $X_9$.
This deformation can be parametrised as 
$\Phi(t,s)=(t^4+s(t^2+1),t^5+s(t^3+t))$, but as germ at the origin
$\vp_s(t)=\Phi(t,s)$ is an immersion for $s\neq 0$.
\end{example}
Bruce and Gaffney call an irreducible  function germ 
$f\colon (\C^2,0)\to (\C,0)$ \textit{irreducible $\cK$-simple}
if all neighbouring irreducible functions fall into finitely
many $\cK$-equivalence classes.
The parametrisation of a curve, which is  irreducible $\cK$-simple,
is $\cA$-simple.
The confining singularities for irreducible plane curve singularities 
are those with Puiseux pairs $(4,9)$ and $(5,6)$. 
All irreducible curves below these ones have only finitely many
$\cK$-orbits, so are therefore  irreducible $\cK$-simple.
The complete list consists $A_{2k}$, $E_{6k}$,  $E_{6k+2}$,
$W_{12}$, $W_{18}$ and $W^{\#}_{1,2q-1}$. In particular, the
list of $\cA$-simple parametrisations coincides with that of 
 irreducible $\cK$-simple functions.

The classification of $\cA$-simple curves was extended to space curves by
Gibson--Hobbs \cite{GH}
and  by Arnol'd \cite{ar2} to  irreducible curves of arbitrary embedding 
dimension and finally to reducible curves
by Kolgushkin--Sadykov \cite{KS}. The lists become rather long.

The other possibility  in the situation of Example \ref{AKexample}
is to change the concept of simpleness for parametrisatrions.
This approach was taken by Zhitomirskii \cite{zh}.
His definition of fully simple singularities involves arcs.

\begin{defn}
An \textit{arc} $F\colon [a,b] \to \R^n$ is said to \textit{represent}
a multi-germ 
\[
\gamma\colon \coprod_{i=1}^r(\R_{(i)},0) \to (\R^n,0)
\]
if the multi-germ  $(F,F^{-1}(0))$ is $\cA$-equivalent to $\gamma$.
Here we assume that  the image of $F$ contains the origin,
and that  the endpoints $F(a)$ and $F(b)$
are different from the origin. 
\end{defn}

\begin{defn}
A multi-germ $\gamma$ of a parameterized curve in 
$\R^n$ is \textit{fully simple} is there exists an arc 
$F\colon [a,b] \to \R^n$ representing $\gamma$ such that 
the  singularities of all arcs in a neighbourhood of $F$
at all points of their images sufficiently close to the origin
belong to finitely many equivalence classes.
\end{defn}

This definition extends in a natural way to complex parametrisations.
It is convenient to represent a reducible curve by a finite number of
arcs.
A nearby fibre in a good representative of the germs of the versal
deformation of a parametrisation  gives a finite collections of
complex arcs. The versal deformation contains representatives
for the isomorphism classes of all neighbouring arcs.
Therefore the simple complex parametrisations, in the sense of Definitiion
\ref{defsimple} are the complex fully simple parametrised curves
of Zhitomirskii \cite{zh}.

\subsection{Stably equivalent parametrisations}
In a deformation of the parametrisation the embedding dimension
can increase. 
Therefore the collection of confining singularities 
depends on the chosen target dimension for the parametrisation
we start with. Two parametrisations which only differ in target
dimension are called \textit{stably equivalent} \cite{ar2}.
A parametrisation is \textit{stably simple}  if all stably equivalent
parametrisations are simple.

\begin{lemma}
A simple parametrisation is stably simple.
\end{lemma} 
\begin{proof}
Suppose a simple parametrisation $\vp\colon (\wl C,\wl 0)\to (\C^n,0)$ 
deforms with higher target dimension into a parametrisation with moduli, 
so there exist a family $\psi_s \colon (\wl C,\wl 0)\to (\C^{n+k},0)$ with 
moduli. For a generic projection $\pi\colon (\C^{n+k},0) \to (\C^n,0)$
the family $\pi\circ\psi_s$ is a deformation of $\vp$. One expects
a generic projection of a singularity to have more moduli than the  
singularity itself, so $\pi\circ\psi_s$ has moduli, contradicting that
$\vp$ is simple. It suffices to prove this for the confining singularities
for stable simpleness. By Theorem \ref{confine} they are the 
curves  $L_{n+2}^n$ of Definition \ref{lnn2}, 
and for them the expectation is indeed true.
\end{proof}

We classify stably simple parametrisations. The Lemma justifies that
we speak only of simple parametrisations and drop the word \lq stable\rq.
We always consider a curve as embedded in $(\C^N,0)$ 
for $N$ large enough, except in the section on plane curves.

\section{Notations}
\subsection{Curves with smooth branches}
\begin{defn}
A curve singularity $C=C_1\cup C_2$ is \textit{decomposable}
if the curves $C_1$ and $C_2$ lie in smooth spaces intersecting
each other transversally in one point, the singular point of $C$.
We write $C=C_1\vee C_2$.
\end{defn}

We write $C\vee L$ for the wedge of $C$ with a smooth branch.

\begin{defn}
The curve $L_n^n=L\vee\dots\vee L\subset \C^n$ is the curve
isomorphic to the singularity consisting of the coordinate axes in 
$\C^n$. The curve $L_{n+1}^n$, $n\geq 2$ is the curve consisting
of $n+1$ lines in $\C^n$ through the origin in general position,
meaning that each subset consisting of $n$ lines span $\C^n$.
\end{defn}

Note that $L_3^2$ is the
plane curve singularity $D_4$. 

Points in projective space are in generic position if each subset
imposes independent conditions on hypersurfaces of each degree
\cite{gr}. The curve $L_{n+2}^n$, which is the cone over $n+2$
points in generic position in $\P^{n-1}$, has $\mu =\delta +2$, if
$n\geq 3$. But the singularity $\wt E_7$, four lines though the
origin, has $\mu =\delta +3$. There exists a curve with the same
tangent cone, having $\mu =\delta +2$; we lift one branch out of the
plane. Let $\wt E_7$ be given by $xy(x-y)(x-\lambda y)=0$.
We take the same first three lines, but parametrise the last one 
as $(x,y,z)= (\lambda t,t,t^2)$. The equations are determinantal:
\begin{equation}\label{l42}
\Rk
\begin{pmatrix}
z  & \lambda (x-y)  & y(x-y)\\
0  & x-\lambda y &    z-y^2
\end {pmatrix}
\leq 1\;.
\end{equation}
We will call this curve $L_4^2$. 
As it is not a complete intersection, there is no deformation from
$\wt E_7$, but there is a deformation of the parametrisation.

The curve $L_3^1$ consists of three smooth branches with common 
tangent. The plane curve $\wt E_8\colon x(x-y^2)(x-\lambda y^2)$
has $\mu =\delta +3$. We can again lift one branch out of the plane
and parametrise $(x,y,z)=(\lambda t^2,t,t^3)$.
Equations are 
\begin{equation}\label{l31}
\Rk
\begin{pmatrix}
z  & \lambda (\lambda-1)y  & \lambda x\\
0  & x-\lambda y^2 &   \lambda z-xy
\end {pmatrix}
\leq 1\;.
\end{equation}
If we lift the line further out of the plane, as $(x,y,z)=(\lambda t^2,t,t^2)$,
the coefficient of $t^2$ in $x$ can be transformed into 1,
and we get the simple curve denoted $J_{2,0}(2)$ by
Fr\"uhbis-Kr\"uger \cite{FK} and denoted $S_3^t$ in \cite{st-c}.
The difference between the curves $S_3^t$ and $L_3^1$
can be seen from the 2-jet of the parametrisation.
Following \cite{zh} we
say that the 2-jet $j^2\vp$ is \textit{planar} if the image of $\vp$
lies modulo terms of third order on a smooth surface.

\begin{defn}\label{lnn2}
The curve $L_{n+2}^n$ is for $n\geq 3$ the curve consisting of 
$n+2$ lines through the origin in generic position in $\C^n$, 
the curve $L_4^2$ is the curve with equations (\ref{l42})
and $L_3^1$ the curve with equations (\ref{l31}).
\end{defn}

\subsection{Notation for singular curves}
We will denote monomial curves by their semigroup, so the curve
$Z_{10}\colon z^2+yx^2=y^2+x^3=0$ of \cite{gi} is $(4,6,7)$.
Plane curves $(2,2k+1)$ are mostly referred to by their name $A_{2k}$.
Also for the monomial curve of minimal multiplicity 
$(k,k+1,\dots,2k-1)$ with $\delta=k-1$ we use a special name $M_k$.
We extend this notation to quasi-homogeneous reducible curves
by writing the exponents of the parameter. The union of curves is indicated by a plus sign. If for some coordinate
function $z_i=\vp_i(t)=0$, we write a dash. For example, the
curve $S_3^t=J_{2,0}(2)$ described above is notated
$(1,\n-,\n-)+(1,2,\n-)+(1,\n-,2)$.  

\subsection{Notation for adjacencies}
The name or symbol denotes both a curve and its parametrisation.
There are two types of adjacencies, for deformations of the 
parametrisation and for deformations of the image curve.
An adjacency, which can be obtained in both ways, is denoted
by the usual arrow $\to$, an adjacency of parametrisations
by an arrow $\rightharpoondown$ or a rotated version, if the 
arrow is written in an other direction, while an arrow
$\rightharpoonup$ is for adjacencies of curves only.

\section{Main results on parametrisations}
With the notations introduced  above we can formulate 
our classification result.

\begin{theorem}\label{thmclass}
The infinite series of curves 
$A_k\vee L_n^n$, $D_k\vee L_n^n$ and
$E_k\vee L_n^n$ and
the sporadic curves $(5,6,7,8)$, $(4,6,7)$,  $(2,3,\n-,\n-)+(\n-,4,5,3)$,
and $(4,5,7)\vee L$ have simple parametrisations. 
Any other simple parametrisation occurs in the versal deformation
of one of these parametrisations.
\end{theorem}

A complete list of simple parametrisations is given in the next section.
In the course of the classification we also determine the
confining singularities, thereby proving (in the complex case)
Conjecture A1 of
Zhitomirskii \cite{zh}. 

\begin{theorem}\label{confine}
The confining singularities for deformations of parametrisations
are the curves $L_n^{n-2}$ from Definition {\rm \ref{lnn2}}.
\end{theorem}

The list of simple parametrisations shows that also
Conjecture B1 of Zhitomirskii \cite{zh} is true:

\begin{cor}
The curve singularities with simple parametrisations are
quasi-homogeneous.
\end{cor}

\section{List of simple parametrisations}
We list the curves together with some adjacencies.
These are by no means all adjacencies, but we rather use them to 
organise the list. We start with the sporadic curves.

\subsection{Sporadic  curves}
For all curves listed the $\delta$-invariant satisfies $\delta\leq 5$.
In each case the most singular curve has $\delta=5$ and an adjacency
of parametrisations and image curves (given by an arrow $\leftarrow$
or $\downarrow$)
is $\delta$-constant, while the other adjacencies lower $\delta$ by one.
\subsubsection{Irreducible curves}\label{unibranch}
There are eight unibranch sporadic curves.
\[
\begin{matrix}
&&&&(5,6,7,8,9)& \leftharpoonup &(5,6,7,8)\\
&&&&\downarrow&&\downharpoonleft\\
 (4,5,6,7) &\leftharpoonup  &  (4,5,6) & \leftarrow &(4,5,7) 
       & \leftarrow &(4,6,7,9) &\leftharpoonup  &    (4,6,7)\\
&&&&&&\downarrow\\
&&&&&&(3,7,8) 
\end{matrix}
\]
\subsubsection{One branch of multiplicity four and a line}
\label {multfour}
\[
\begin{matrix}
&&(4,5,6,7)+(\n-,\n-,1,\n-) & \leftarrow &(4,5,7)\vee L \\
&&\downarrow&&\downarrow\\
(4,5,6,7)\vee L &\leftharpoonup  &  (4,5,6,7)+(\n-,\n-,\n-,1)  & \leftarrow
             &(4,5,6)\vee L 
\end{matrix}
\]
\subsubsection{One branch of multiplicity three and a cusp}
\label{A2M3}
\[
\begin{matrix}
A_2\vee M_3 &\leftharpoonup  &  (2,3,\n-,\n-)+(\n-,5,4,3)  & \leftarrow
             &(2,3,\n-,\n-)+(\n-,5,4,3)  
\end{matrix}
\]
\subsubsection{Two cusps and a line}
\label{twocuspsline}
\[
\begin{matrix}
&&\begin{gathered}(2,3,\n-,\n-)+(\n-,3,2,\n-)\\{}+(\n-,\n-,\n-,1) 
\end{gathered}
    & \leftarrow &
\begin{gathered}(2,3,\n-,\n-)+(3,\n-,2,\n-)\\{}+(\n-,\n-,\n-,1)\end{gathered}  \\
&&\downarrow&&\downarrow\\
A_2\vee A_2\vee L &\leftharpoonup  &  
\begin{gathered}(2,3,\n-,\n-)+(\n-,\n-,3,2)\\{}+(\n-,1,1,\n-) \end{gathered} 
& \leftarrow  
&\begin{gathered}(2,3,\n-,\n-)+(\n-,\n-,3,2)\\{}+(1,\n-,1,\n-) \end{gathered}
\end{matrix}
\]
 \subsubsection{The union of two $A_k$-singularities}
\label{twoAk}
 
\[
\begin{matrix}
A_2\vee A_4&\leftharpoonup  & 
(2,3,\n-)+(\n-,5,2)  &&\qquad (2,3,\n-)+(2,\n-,3) \\
\downarrow&&\downarrow&&\qquad \downharpoonleft\\
A_2\vee A_3&\leftharpoonup&(2,3,\n-)+(\n-,\n-,1)+(\n-,2,1)& 
& \qquad (2,3,\n-,\n-)+(2,\n-,3,4)\\
\downharpoonright&&\downharpoonright&&\swarrow\\
A_2\vee A_2&\leftharpoonup  &  (2,3,\n-)+(\n-,3,2) & \leftarrow
             &  (2,3,\n-)+(\n-,2,3) \qquad
\end{matrix}
\]
\subsubsection{Other sporadic curves}
\[
\begin{matrix}
(3,4,5,\n-)+(1,\n-,\n-,2) &\leftharpoonup &(3,4,5)+(1,\n-,\n-) \\
\downarrow&&\downharpoonleft\\
(1,\n-,\n-)+(1,2,\n-)+(1,\n-,2) &\leftarrow &(2,5,\n-)+(1,\n-,2) 
\end{matrix}
\]

\subsection{Infinite series, of the form $C\vee L^k_k$}
All singularities in this part of the list are related to $A_k$, $D_k$
or $E_k$. We have therefore series of series and individual series.
A series is of the form $C\vee L_k^k$ with $C$ indecomposable. 
Here we allow $k=0$
and interpret $L_0^0$ as point, so  $C\vee L_0^0$ is just the curve
$C$ itself.
We list below only the indecomposable curves $C$. The only curve not of 
this form is the simplest of all, the totally decomposable curve
$L_n^n$. This curve is singular if $n\geq 2$,
with $L_2^2=A_1$. 
We include $L_n^n$ by including $A_1$ in the list, even though it is 
decomposable.

\subsubsection{Indecomposable curves of type $E$ and deformations}
\label{typeE}
\[
\begin{matrix}
(3,4,5)  & \leftharpoonup & E_6\colon(3,4)
      & \leftarrow & \quad(3,5,7) \quad& \leftharpoonup &E_8\colon(3,5)\\
&&&&\downarrow&&\downarrow\\
&&&&\hphantom{\quad(3,5,7) \quad}\llap{$(2,3,\n-)+(1,\n-,2)$}& \leftharpoonup &E_7\colon(2,3)+(1,\n-)
\end{matrix}
\]
\subsubsection{Deformations of $E_k\vee L_{n-2}^{n-2}$}
Here $n$ is the embedding dimension, which has to satisfy $n\geq3$.
From $E_8$ and $E_6$ we get
\begin{align*}
(3,5,7)\vee L_{n-3}^{n-3} &+ (\n-,\n-,1,\dots,1) \\ 
&\downharpoonleft\\
(3,4,5)\vee L_{n-3}^{n-3} &+ (\n-,1,\n-,1\dots,1) \\ 
&\downarrow\\ 
(3,4,5)\vee L_{n-3}^{n-3} &+ (\n-,\n-,1,\dots,1)   
\end{align*}
and from $E_7$ 
\[
A_2\vee L_{n-2}^{n-2} + (1,\n-,1,\dots,1) 
\leftharpoonup  
A_2\vee L_{n-2}^{n-2}+(1,\n-,2,\dots,2)   
\]
\subsubsection{Indecomposable curves of type $A$}

\begin{gather*}
A_1\colon (1,-)+(-,1)\\
A_{2k-1}\colon (1,-)+(1,k) \leftarrow A_{2k}\colon (2,2k+1)
\end{gather*}

\subsubsection{Deformations of $D_k\vee L_{n-2}^{n-2}$}\label{dvl}
\begin{gather*}
L_n^{n}+(1,\dots,1)\\
A_{2k}\vee L_{n-2}^{n-2}+ (\n-,1,1\dots,1)\leftarrow
A_{2k-1}\vee L_{n-2}^{n-2}+ (\n-,1,1\dots,1)
\end{gather*}
Here $n\geq 2$ is again the embedding dimension. For $n=2$
the curves are the plane curves $D_4$, $D_{2k+3}$ and $D_{2k+2}$.

\section{Classification}
The proof  of Theorems \ref{thmclass} and \ref{confine}
proceeds by classifying all parametrisations which do not
deform into  a parametrisation of a curve  $L_n^{n-2}$.
The result is that these do not have moduli. Furthermore
we show that all other parametrisations do deform into $L_n^{n-2}$.
Therefore singularities of the list can only deform into
other singularities of the list, implying simpleness.

We start by describing large classes of  parametrisations, which are not
simple. From them we derive restrictions on the multiplicities 
of the irreducible components of curves with simple parametrisation.

\subsection{Some adjacencies}
\subsubsection{Every parametrisation of a curve $C=C_1\cup C_2$
deforms into $C_1\vee C_2$} Parametrise  $C_1$ with $\vp^{(1)}
\colon \wl C_1 \to\C^n$ and  $C_2$ with $\vp^{(2)}
\colon \wl C_2 \to\C^n$, and consider the curve as lying in $\C^{2n}$.
The  parametrisation, given by $(\vp^{(1)}, 0)$ and 
$(\vp^{(2)},s\vp^{(2)})$ has for $s\neq 0 $ image $C_1\vee C_2$.

If  a curve $C$ with simple parametrisation is reducible, 
and can be written as union $C'\cup C''$,
then both $C'$ and $C''$ have a simple parametrisation.

\subsubsection{A parametrisation of an irreducible curve of multiplicity
$m$ deforms into the monomial curve $M_m$}
We may assume that we have a parametrisation $\vp \colon \wl C
\to \C^m$ with first component $z_1=\vp_1(t)=t^m$. Now deform
$z_1=t^m$, $z_i=\vp_i(t)+st^{m+i-1}$ for $i\geq 2$.
\subsubsection{$M_m$ deforms into $M_{m_1}\vee\dots\vee
M_{m_k}$ for any partition $(m_1,\dots,m_k)$ of $m$} 
A description in terms of equations is given in \cite[p.~199]{st-u}.
A simple argument in terms of the parametrisation is the following.
The curve $M_{m}$ is
a special hyperplane section of the cone over the rational
curve of multiplicity $m$ and is resolved by one blow-up.
Now deform the smooth strict transform such that it
intersects the exceptional divisor in $k$ points
with multiplicities given by the
partition $(m_1,\dots,m_k)$ and blow down again.
\subsubsection{$A_2\vee L \to A_3$}
This is a special case of the adjacency  $A_k\vee L \to D_{k+1}$
(in fact $D_3=A_3$), which
can be inferred from the formulas of \cite[p.~1040]{FK}, but is missing
in \cite[Diagram 4]{FN}.
Consider the deformation
\[
\Rk
\begin{pmatrix}
x^k & y & z\\
y & x & s
\end{pmatrix}
\leq 1\;.
\]
One branch is $(0,0,t_1)$ and for even $k$ the second
branch is $(t_2^2, t_2^{k+1},st_2^{k-1})$.

\subsection{First consequences}
The curve $A_3\vee A_3$ 
is not simple, as it deforms to $L_4^2$ (with constant 
$\delta=5$); just rotate some lines in the plane spanned by the
tangent lines of both $A_3$-singularities.  As $A_2\vee L \to A_3$, and 
$M_6 \to A_2\vee A_2 \vee L_2^2$, the last two curves are also 
not simple. Also $M_5\vee L$ and $M_4\vee L_2^2$ are not simple.

We conclude that the parametrisation of an irreducible curve
of multiplicity at least 6 is not simple. 
A simple parametrisation with at least four
branches has at most one singular component, of multiplicity
at most three.
A (sporadic) simple curve has at most two singular components
($A_2\vee A_2\vee A_2$ is not simple), and the multiplicity is at
most 5.

\subsection{Irreducible curves}
We may assume that the parametrisation has the form
$x_i=\vp_i(t)$, $i=1,\dots,k$, with $v(\vp_i)<v(\vp_j)$ for
$i<j$, $v(\vp_i)$ being the order in $t$ of $\vp_i$.
We also can achieve that $v(\vp_j)$ does not lie in the
semigroup generated by the $v(\vp_i)$ with $i<j$.

A parametrisation of a curve of multiplicity 5 (that is, 
$v(\vp_1)=5$) is not simple if $v(\vp_4)>10$: deform into
$L_5^3$ by perturbing $\vp_1$, $\vp_2$ and $\vp_3$ with
terms divisible by $t^5-s$ and $\vp_j$, $j\geq4$, with terms
divisible by $(t^5-s)^2$. Similarly, the curve deforms to
$L_4^2$ if $v(\vp_3)\geq8$.
A parametrisation of a curve of multiplicity 4  is not simple if  
$v(\vp_3)\>8$. A multiplicity 3 curve deforms into
$L_3^1$, a curve with planar 2-jet, if $v(\vp_3)>9$.
Irreducible double points are simple.

This leaves only a few possibilities for simple parametrisations.
Their normal forms can be computed with standard methods;
they can be found in the paper by Ebey \cite{eb}.

\begin{lemma}
The curve $(5,6,7,9)$ is not simple, as
$(5,6,7,9)\to L_3^1$.
\end{lemma}
\begin{proof}
Consider the deformation
\[
\vp_s(t)=((t^3-s)t^2,(t^3-s)^2,(t^3-s)^2t,(t^3-s)^3)\;.
\]
The parametrisation satisfies the equations
$w=sz-x\equiv 0 \mod (t^3-s)^3$, so
for $s\neq 0$ the 2-jet of $\vp_s(t)$ is planar.
\end{proof}

\begin{prop}\label{prop5678}
The parametrisations of the curves $(5,6,7,8)$ and $(4,6,7)$ are simple. 
They deforms into the other unibranch sporadic curves of\/
{\rm\ref{unibranch}}  and the irreducible triple points
of\/ {\rm\ref{typeE}}.
\end{prop}
\begin{proof}
As explained above, we now only show that there is no deformation
to $L_{n}^{n-2}$. It suffices to consider the ones with $\delta=5$.
A deformation of the parametrisation of $(5,6,7,8)$ or $(4,6,7)$ 
to $L_4^2$ or $L_3^1$ is $\delta$-constant, so also a deformation of
the curve. The curves $(5,6,7,8)$ and $(4,6,7)$ are Gorenstein, but
$L_4^2$ and $L_3^1$ not. Therefore such a deformation does not
exist.

The adjacencies of \ref{unibranch} and \ref{typeE} 
are easily established.	
\end{proof}

\subsection{Curves with one singular component of
multiplicity three or four}

\subsubsection{Multiplicity four}
The curve $(4,6,7,9)$ deforms into the (simple) curve
$J_{2,0}(2)=S_3^t$ consisting of three tangent lines with non-planar
2-jet and therefore $(4,6,7,9)\vee L$ deforms into $L_4^2$ and
is not simple. If the line in the curve $(4,5,6)\cup L$ is not 
transverse to the Zariski tangent space of $(4,5,6)$, then the
curve deforms into $L_5^3$.
This leaves  $(4,5,7)\vee L$, $(4,5,6)\vee L$ and curves of
the type $(4,5,6,7)\cup L$. The classification of the latter curves
follows from the general results of \cite[2.2]{st-c}: the
isomorphism type depends on the osculating space of $M_4$,
to which the line is tangent, and the line can be taken to be 
a coordinate axis, except in the most degenerate case, that the
line is tangent to the tangent line of the curve.  The curve
$M_4$ deforms into $D_4$, with tangent plane the 
$(x_1,x_2)$-plane, so if the line is tangent to this plane, there
is a deformation to $L_4^2$. 

\begin{prop}\label{prop457}
The parametrisation of the curve $(4,5,7)\vee L$ is simple. 
It deforms into the other sporadic curves of\/
{\rm\ref{multfour}}. 
\end{prop}

\begin{proof}
If $(4,5,7)\vee L\to L_4^2$, then $(4,5,7)$ deforms into three
smooth branches, tangent to plane containing the line $L$. Projection
along $L$
onto $\C^3$ gives a $\delta$-constant
deformation to the space curve $J_{2,0}(2)$
consisting of three smooth branches with common tangent.
According to the tables in \cite{FK} such a deformation does
not exist. To be self-contained we give a proof along the lines of
the proofs in \cite{zh}.

So suppose $(4,5,7)\to J_{2,0}(2)$.
The parametrisation has the form
\[
\vp_i(t,s)=(t-a_s)(t-b_s)(t-c_s)\psi_i(t,s)
\]
for $i=1,2,3$.
The images of the germs $(t,a_s)$, $(t,b_s)$ and $(t,c_s)$ are 
tangent to a line
$L_s$, which has a limiting position for $s\to 0$. By a coordinate
transformation we may suppose that the line $L_s$ is constant. 
It is given by two linearly independent equations of the form
$Az_1+Bz_2+Cz_3=0$. This implies that 
\[
A\vp_1(t,s)+B\vp_2(t,s) + C\vp_3(t,s) \equiv 0
 \mod
  (t-a_s)^2(t-b_s)^2(t-c_s)^2
\]
for all $s$. Specialising to $s=0$ leads to the 
equation 
$
At^4+Bt^5+Ct^7\equiv 0 \mod t^6
$,
from which we conclude that $A=B=0$. But then there is only one
linear equation.

If $(4,5,7)\vee L\to L_3^1$, then $(4,5,7)$ deforms into two
smooth branches with as common tangent the line $L$. Projection
onto $\C^3$ gives a deformation to the space curve consisting
of two cusps with common tangent, as $L_3^1$ has planar 2-jet.
But this curve is at least $Z_9$ with $\delta=5>4=\delta(4,5,7)$.
\end{proof}

\subsubsection{Multiplicity three}
Let $C_3$ be an irreducible curve of multiplicity 3. A 
parametrisation-simple
union of $C_3$ and $n$ lines has embedding dimension 
at least $n+2$, for otherwise it deforms into $L_{n+3}^{n+1}$.
The $n$ lines form an $L_n^n$: 
for $n=1$ this is trivial; if $n=2$  and the lines are $A_{2k-1}$
with $k>1$, then the parametrisation deforms into the non-simple
$A_{2k-1}\vee A_4$, as $M_3$ deforms into $A_4$; finally, if
$n>2$ and the lines deform into $L_n^{n-1}$, then the parametrisation
deforms into $L_{n+1}^{n-1}$, as the parametrisation of
$C_3$ deforms into a
smooth branch with arbitrary tangent.

The curve $E_{12}(2)=(3,7,8)$ deforms into 
$J_{2,0}(2)=S_3^t$, so $(3,7,8)\vee L$ is not simple.

\begin{prop}\label{prope8}
The curves $E_6\vee L_n^n$ and $E_8\vee L_n^n$
have simple parametrisations.
\end{prop}

\begin{proof}
As $E_8\to E_6+A_1$ it suffices to show simpleness for $E_8\vee L_n^n$.
We have to exclude deformations
of the parametrisation into an  $L_{k+2}^k$. If $n-m$ of the $n$ 
deformed lines do not pass through the singular point of $L_{k+2}^k$,
then there is also a deformation $E_8\vee L_n^n\rightharpoondown
L_{k+2}^k$.  So we may assume that $m=n$, and that the $n$ lines
are not deformed at all. The only possibilities for $k$ are therefore
$k=n$ and $k=n+1$.

If  $E_8\vee L_n^n\rightharpoondown
L_{n+2}^n$, then $E_8$ is deformed into two smooth branches
tangent to the space spanned by  $L^n_n$. Projection onto the plane 
of the $E_8$ gives a deformation of the parametrisation into 
the union of two plane curves of multiplicity two, which is impossible.

If $E_8\vee L_n^n\rightharpoondown
L_{n+3}^{n+1}$, then projection onto the plane of the  $E_8$ 
gives a deformation of $E_8$ into three tangent smooth branches,
which is also impossible, as this would increase $\delta$.
\end{proof}

For the curves $C_3$ of type 
$E_8(1)=(3,5,7)$ and $E_6(1)=M_3=(3,4,5)$
with $(C_3\cdot L_n^n)>1$
we look at the 1-dimensional intersection $T$ of the 
Zariski tangent spaces of the singular curve and $L_n^n$.
If for $(3,5,7)$ the line $T$ lies in the $(x_1,x_2)$-plane,
then the curve deforms into $L_{n+3}^{n+1}$, as $(3,5,7)
\to D_4$. Otherwise there is a transformation bringing $T$
to the $x_3$-axis. 
In $L_n^n$ the line $T$ is in the direction $(1,\dots,1,0,\dots,0)$.
The curve is indecomposable if and only if there are no zeroes.
We transform to a different normal form, where
the line $T$ is a coordinate axis.
The resulting curve is a deformation
of $E_8\vee L_n^n$.

The curve $(3,4,5)$ deforms into $A_3$ with the $x_1$-axis
as tangent line, so if $T$ is this axis, then the curve is not simple
if $n>1$: it deforms to $L_{n+2}^n$. For $n=1$ the curve
is simple: if the 2-jet of the parametrisation of $L$
has image $T$, then it is the curve $W_9$, which is a 
$\delta$-constant deformation of $Z_{10}=(4,6,7)$. There is
also a curve with $\delta=4$, with $L=(t,0,0,t^2)$.
For the other cases the intersection multiplicity 
$(M_3\cdot L_n^n)$ is equal to 2, and \cite[2.3]{st-c}
applies. If $T$ does not lie in the $(x_1,x_2)$-plane, then
the curve is  a deformation of $E_6\vee L$, otherwise of
$E_8\vee L$, under the deformation of the parametrisation
$(t^3,st^4,t^5,0,\dots,0)$.

\subsection{Two singular components}
As every parametrisation of an irreducible curve other than
$A_2$ deforms into $A_3$, one component has to be $A_2$.
The curve $A_2\vee A_5$ is not simple, as it deforms into
$L_3^1$. This implies that the other singular component
is $M_3$, $A_4$ or $A_2$.

\subsubsection{$A_2\,\cup\, M_3$}
The embedding dimension is at least
4. Unless the curve is  $A_2\vee M_3$ we let
$T$ be the intersection line of the Zariski tangent spaces
of the components. As $A_2$ deforms into $A_1=L_2^2$ the
curve is not simple, if $T$ is the tangent line of $M_3$, by what
was said above for $M_3\cup L_2^2$.
Otherwise $(A_2\cdot M_3)=2$ and $T$ may be
taken as coordinate axis.  There are four curves to consider.

\begin{lemma}
The curve $(2,3,\n-,\n-)+(4,\n-,5,3)$ and
$(2,3,\n-,\n-)+(5,\n-,4,3)$ are not simple, as they deform to
$L_3^1$.
\end{lemma}
\begin{proof} 
The first curve deforms into the second. For that case we
deform  the cusp 
into a smooth branch by $(t^2,t^3,0,2st)$
and  $M_3$ into $A_3$
by 
\[
((t^2-s^2)^2t, 0, (t^2-s^2)^2, (t^2-s^2)(t+2s))\;.
\]
The $A_3$ lies on the smooth surface
\[
12xs^6-3w^2s^4-xw+z^2+2zws^2+12zs^8=0\;.
\]
The parametrisation of the smooth branch satisfies this equation
modulo terms of degree 3. The intersection number of the branch
with  $A_3$  is 3, so we have three smooth
tangent branches with $\delta=5$, which is $L_3^1$. 
\end{proof}

\begin{prop}\label{propA2M3}
The curves $(2,3,\n-,\n-)+(\n-,4,5,3)$
and $(2,3,\n-,\n-)+(\n-,5,4,3)$ are simple.
\end{prop}

\begin{proof}
Suppose first that such a curve deforms to $L_4^2$.
The component $M_3$ deforms only to three smooth branches
spanning 3-space, so both components have to deform to two
smooth branches, and the two smooth branches, into which $M_3$
deforms, are tangent to the plane of the cusp. Then the last component
of the parametrisation has the form
$\vp_4(t,s)=(t-a_s)^2(t-b_s)^2\psi_i(t,s)$, but $\vp_4(t,0)=t^3$.

If the curve deforms to $L_3^1$, then the cusp deforms
into a smooth branch and $M_3\to A_3$ is a $\delta$-constant 
deformation. The equation $z_1=0$ of $M_3$ deforms into 
$z_1+sf(z_1,\dots,z_4)=0$ and the first component of
the parametrisation of $A_2$ is $\vp_1(t,s)=t^2+s\psi_i(t,s)$.
The intersection multiplicity of the smooth branch and $A_3$ is at
most the order in $t$ of $\vp_1+sf(\vp_1,\dots,\vp_4)$, so at most 2.
Therefore the three smooth branches form the simple singularity
$J_{2,0}(2)$, with $\delta=4$.
\end{proof}

\subsubsection{$A_2\,\cup\, A_4$}
Such a curve is not simple if the tangent line
of $A_4$ lies in the plane of the cusp $A_2$, as it then
deforms to $L_4^2$. If the tangent line of the cusp lies
in the plane of the $A_4$, then there is deformation to 
$L_3^1$. The curve $T_9=(2,3,\n-)+(\n-,5,2)$ is a deformation
of $Z_{10}=(4,6,7)$. It deforms into $A_2\vee A_4$.
The parametrisation of  $A_2\vee A_4  \vee L$ is not simple.

\subsubsection{$A_2\,\cup\, A_2$}
All possibilities
for the intersection line $T$ yield simple curves.
The curve $Z_{10}$ deforms into $Z_9=(2,3,\n-)+(2,\n-,3)$.
A deformation of the parametrisation gives the curve
$Z_9(1)=(2,3,\n-,\n-)+(2,\n-,3,4)$ with $\delta=4$. It deforms 
$\delta$-constant
into $T_7^*=(2,3,\n-)+(3,\n-,2)$ and then into
 $T_7=(3,2,\n-)+(3,\n-,2)$. By a deformation of the
parametrisation we obtain  $A_2\vee A_2$.
The curves here and of the previous paragraph are listed in \ref{twoAk}.

\subsubsection{$A_2\cup A_2\cup L$}
The curve $Z_9(1)\vee L$ deforms into $L_3^1$. The curve
consisting of $A_2\cup A_2$ and a smooth branch is not
simple if the smooth branch is tangent to the plane spanned
by the tangent lines of the cusps, for then there  is again a
deformation to $L_3^1$.
The branch is also not tangent to the plane of one of the cusps,
as $A_2\vee D_4$ is not simple, deforming into $L_4^2$.

The singularity $T_7^*\vee L$ is a deformation of
$W_8^*\vee L=(4,5,7)\vee L$,
as $W_8^*\to T_7^*$ \cite{FK}: use the parametrisation
$(t^2(t-s)^2,  t^3(t-s)^2, t^4(t-s)^3)$.

The curves $A_2\cup A_2\cup L$
in to which the parametrisation of 
$T_7^*\vee L$ deforms are listed in \ref{twocuspsline}.

\subsection{At most one component of multiplicity two}
If there are only smooth branches it can happen that some branches
have the same tangent line. As $A_3\vee A_3$ is not simple, this 
can happen only for one direction. The curve $J_{2,0}(2)$ consisting
of three smooth branches is a deformation of $J_{2,1}(2)=
(2,5,\n-)+(1,\n-,2)$.  The curve $J_{2,0}(2)\vee L$
deforms into $L_4^2$. So if the curve has at least four branches,
only two of them can be tangent.
\subsubsection{Curves containing  an $A_k$, $k\geq 3$}
As $A_3\vee L_{n}^{n-1} \to L_{n+2}^{n} $, the lines in a curve,
consisting of an $A_k$ ($k\geq 3$) and $n$ lines,  form an 
$L_{n}^{n}$.
The intersection of the space spanned by this $L_{n}^{n}$ 
with the tangent
plane of the $A_k$ is at most 1-dimensional. If it is a line, this line is
not tangent to the $A_k$, for otherwise there is  a deformation of the
curve into $ L_{n+2}^{n} $. So we can take the line
to be a coordinate axis, and get the normal form
listed above (\ref{dvl}), see also \cite[Example 2-14]{st-c}.
Note that for $n=1$ we have $D_{k+3}$. Any curve of this type
is a deformation of $D_{k+3}\vee L_{n-1}^{n-1}$.

\begin{prop}
The curves  $D_{k+3}\vee L_{n}^{n}$ have simple parametrisations.
\end{prop}
\begin{proof}
It suffices to prove the statement for $D_{2m+3}\vee L_{n}^{n}$. 
Again
we have to exclude a deformation to $L_{n+2}^{n}$ or 
$L_{n+3}^{n+1}$.
In the first case the deformed line of $D_{2m+3}$ does not pass 
through the singular point, and in the second case we can assume
that this line and $L_{n}^{n}$ are unchanged. In both cases the 
$A_{2k}$ in $D_{2k+3}$ deforms into two smooth branches, whose
projection onto the plane is singular or always tangent to the line
in $D_{2k+3}$, again impossible.
\end{proof}

\subsubsection{Curves containing  an $A_2$}
If two smooth branches have the same tangent, then there are 
no more smooth branches ($A_3\vee A_2\vee L$ is not simple).
The curve $A_2\vee A_5$ is not simple, as it deforms to $L_3^1$.
For $A_2\vee A_3$ the smooth curves cannot be tangent to the plane
of the cusp: there would be a deformation to $L_4^2$. 
The tangent line of the cusp cannot lie in the plane of $A_3$, 
otherwise there is a deformation to $L_3^1$.
The curve
$T_8=(2,3,\n-)+(\n-,\n-,1)+(\n-,2,1)$ is a deformation of $T_9$.

As $A_2$ deforms by deforming the parametrisation into
a smooth branch with arbitrary tangent, the $n$ lines
in a curve containing $A_2$ form a $L_{n}^{n}$. 
Let $T$ be the intersection
of the tangent plan of the $A_2$ with the  space spanned by the 
$L_{n}^{n}$.
If the curve is indecomposable, then $T$ is a line. If $T$ is transverse
to the cusp, then we get the same type of normal form
as for higher $A_k$.  But $T$ can also be tangent to the cusp.
For $n=1$ we have $E_7$ and, by bending the line out of the 
plane, also $E_7(1)$. If there are more lines, and the cusp
is tangent to one of the lines of $L_{n}^{n}$, then we have
$E_7\vee L_{n-1}^{n-1}$, $E_7(1)\vee L_{n-1}^{n-1}$ and also
curves obtained by bending the line out of the plane in the direction
of  $ L_{n-1}^{n-1}$. 
If the cusp is not tangent to one of the lines, we take a normal
form where $T$ is a coordinate axis.

All curves considered here are deformations of $E_7\vee L_{n-1}^{n-1}$.
This curve is simple, as it occurs in the versal deformation
of the parametrisation 
of  the curve $E_8\vee L_{n-1}^{n-1}$, which is simple by
Proposition \ref{prope8}.

\section{Plane curve singularities}
In this section we show that in the case of plane curves 
a parametrisation is simple if and only if its image is a simple
curve. This fact was already observed by Zhitomirsky \cite{zh}
as result of the classification. Here we give a direct argument.
It is based on the characterisation of simple plane curve
singularities given by Barth, Peters and Van de Ven
\cite[Section II.8]{BPV}.

\begin{theorem}
A plane curve singularity is simple if and only if
its multiplicity is at most three and in each step of the embedded
resolution the multiplicity of the (reduced) total transform
is at most three.
\end{theorem}

\begin{proof}
If there is a point on the total transform of multiplicity
at least four, then by a deformation of the parametrisation 
of the curve we can achieve that it is an ordinary multiple point.
Then the blown-down deformed curve has moduli,  as a 
trivialising coordinate transformation downstairs would lift
to one of the ordinary multiple point on the blow-up.

For the converse we use a formula of Wall for the modality (for
right equivalence) in terms of the multiplicity sequence
of plane curve singularities  \cite[Theorem 8.1]{wa-no}:
\[ 
\moda (C)=\sum_P {\textstyle\frac12}(m_P-1)(m_P-2)-r-s+2\;,
\]
where the sum runs over all infinitely near points in a large
enough embedded resolution, $r$ is the number of branches and
$s$ the total number of satellite points.
If the multiplicity of the singularity is two, then $\moda(C)=0$: if
$r=1$ there is at least one satellite point. For multiplicity
three the strict transform has no point of multiplicity three.
If $r=2$ there is again at least one satellite point.
In the case of one branch, if the strict transform on the first 
blow-up is smooth, there are two satellite points. The remaining
possible multiplicity sequence is $(3,2,1,1,\dots)$ with
two satellite points. So again $\moda(C)=0$.
\end{proof}

\begin{cor} 
The  parametrisation of a plane curve is simple if and only if the 
curve is simple.
\end{cor}

\begin{proof}
For plane curves any deformation of the parametrisation gives
a deformation of the image, so simpleness of the image
implies simpleness of the parametrisation. Conversely, if the curve is 
not simple, then by the above proof the adjacency to a 
singularity with moduli can be realised by a deformation of
the parametrisation.
\end{proof}

We classify the possible multiplicity sequences.
They are given in Table \ref{multsequence}.
As the singularities in question have no moduli for
right equivalence, it suffices to find one parametrisation for each
sequence. This can be done using an explicit description of the charts
of the blow-up.

\begin{table}
\caption{Multiplicity sequences for simple plane curve}
\label{multsequence}
\newlength{\flen}\setlength{\flen}{\fboxrule}
\addtolength{\flen}{\fboxsep}
\def\myfbox#1{\fbox{\begin{tabular}{c} #1\end{tabular}}}
\def\mytbox#1{\begin{tabular}{c} #1\end{tabular}}
\def\mytfbox#1{\hspace{\flen}%
\begin{tabular}{c} #1\end{tabular}\hspace{\flen}}
\def\lpl{--- $\cdots$ ---} \def\lp{--- $\cdots$}
\begin{align*}
A_{2k-1}\colon \qquad&  
\myfbox{1\\1}%
\mytbox{\lpl\\ \lpl}%
\myfbox{1\\1}%
\mytbox{--- 1 \lp\\--- 1 \lp}\\
A_{2k}\colon \qquad &   
\mytfbox{2}\mytbox{\lpl}\mytfbox{2}\mytbox{--- 1 \lp}
\\[2mm]
D_{2k}\colon\qquad&
\myfbox{1\\[\flen]1\\1}
\mytbox{---\\[\flen]---\\---}
\begin{tabular}{c@{\hspace{\flen}}l}
\mytfbox{1}&\mytbox{\lp}\\[\flen]
\myfbox{1\\1}&\mytbox{\lpl\\ \lpl}
\myfbox{1\\1}
\mytbox{--- 1 \lp\\--- 1 \lp}
\end{tabular}
\\
D_{2k+1}\colon\qquad&
\myfbox{1\\[\flen]2}
\mytbox{---\\[\flen]---}
\begin{tabular}{c@{\hspace{2\flen}}l} 
\mytfbox{1}&\mytbox{\lp}\\
\mytfbox{2}&\mytbox{\lpl}
\mytfbox{2}
\mytbox{--- 1 \lp}
\end{tabular}
\\[2mm]
E_6\colon \qquad&  
\mytfbox{3}
\mytbox{---}
\mytfbox{1}
\mytbox{--- 1 \lp}
\\
E_7\colon \qquad&  
\myfbox{1\\2}
\mytbox{---\\ ---}
\myfbox{1\\1}
\mytbox{--- 1 \lp\\--- 1 \lp}
\\
E_8\colon \qquad&  
\mytfbox{3}%
\mytbox{---}%
\mytfbox{2}%
\mytbox{--- 1 \lp}
\end{align*}
\end{table}

Using deformations on the blow up we can also easily establish that
the confining singularities are $\wt E_7\colon x^4+ax^2y^2+y^4=0$ 
and $\wt E_8\colon x^3+axy^4+y^6=0$. For instance, if
the  strict transform has a point
of multiplicity three lying on an exceptional curve, then we deform it 
into an ordinary triple point. Blowing down the exceptional curve 
gives a singularity of type $\wt E_7$, which we can move off the
exceptional curve,  resulting in a deformation of the original singularity
into $\wt E_7$.

\section{Simple curves}
For plane curves we showed without using the classification that the
curves with simple parametrisation are exactly the simple curves for
contact equivalence and even right equivalence of the defining equations.

Also for space curve singularities (in $\C^3$) both concepts
of simpleness coincide, as a comparison of the lists of Giusti
\cite {gi} and Fr\"uhbis-Kr\"uger \cite{FK} with the space curves 
in our list shows; in fact, the comparison of the lists of simple curves
with the list of  Zhitomirskii \cite{zh} exposes some inaccuracies there,
like the inclusion of the confining singularity $T_{10}^*\colon 
xy=x^3+y^6+z^2=0$  \cite[I \S9.8]{AGV}, which deforms into $\wt E_8$.
In Table \ref{tablespacecurves} in the Appendix  we list the
indecomposable simple curves. 
The list of confining singularities for flat deformations
of the curve is  longer than for parametrisations, for complete
intersections see \cite[I \S9.8]{AGV} and for determinantal curves
\cite[Table 1]{FK}.

The minimal $\delta$-invariant for a confining singularity for
parametrisations is $\delta=5$. Therefore the list of all simple
parametrisations contains all curves with $\delta\leq 4$. By the
semi-continuity of $\delta$ we find the following corollaries of the
classification.

\begin{cor}
Every curve singularity with $\delta\leq4$ has a simple parametrisation
and it is also simple as curve.
\end{cor}

\begin{cor}
A parametrisation of a curve singularity with $\delta=5$ is simple if and only if the curve is simple.
\end{cor}

\begin{prop}
The sporadic curves with simple parametrisations are also
simple as curve.
\end{prop}
\begin{proof}
A sporadic curve has $\delta\leq5$.
\end{proof}

This partly explains the coincidence of lists. The series of simple
parametrisations are closely related to  $A_k\vee L$, 
$D_k\vee L$ and $E_k\vee L$. 
In fact, this holds in any embedding dimension.
We expect that our list gives the simple singularities.

\begin{conj}
The simple reduced curve singularities are exactly those with
simple parametrisation.
\end{conj}

This implies in particular a negative answer to the 
old unsolved problem whether  rigid reduced
curve singularities exist. The deformation theory of curve
singularities of large codimension is complicated. There
exist non-smoothable curves. They are not simple: the argument 
that they are not smoothable, is that the
number of moduli is larger than the (computable)
dimension of a smoothing component, cf.~\cite{gr}.

\begin{prop}
The curves $L_n^n$, $A_2\vee L_k^k$, $A_3\vee L_k^k$
and $L_{n+1}^n\vee L_k^k$ are simple.
\end {prop}
\begin{proof}
These are the curves with $\delta-r+1\leq 1$ \cite{gr}, and 
$\delta-r+1$ is upper semi-continuous \cite{BuGr}.
\end{proof} 

Also the curves with  $\delta-r+1=2$ are classified, see \cite{st-c}.
The ones with moduli are not Gorenstein, so the Gorenstein curves
$A_{2}\vee L_{n-2}^{n-2}+ (\n-,1,1\dots,1) $
and $A_{3}\vee L_{n-2}^{n-2}+ (\n-,1,1\dots,1)$
are also simple.

\appendix
\section{Simple space curves}

In Table \ref{tablespacecurves} we list the 
indecomposable simple spaces curves
together with their names in the classifications by
Giusti \cite{gi} and Fr\"uhbis-Kr\"uger \cite{FK}. The equations 
are computed to agree with the parametrisations.
The decomposable simple space curves are $A_k\vee L$, 
$D_k\vee L$ and $E_k\vee L$.

\begin{table}
\caption{Indecomposable simple space curves}
\label{tablespacecurves}
$
\begin{array}{lll}
\toprule
\mbox{type}\qquad  & \mbox{parametrisation} & \mbox{equations} \\
\midrule
Z_{10}    &    (4,6,7)  &  y^2-x^3,\; z^2-yx^2\\
Z_9  &  (2,3,\n-)+(2,\n-,3)  &   y^2-x^3,\; z^2-x^3 \\ [1mm]
W_9  &  (3,4,5)+(1,\n-,\n-)  & y^2-xz,\; z^2-yx^2    \\ 
W_8^*  &(4,5,7)  &  \scalefont{0.80}{\arraycolsep=3pt
\begin{pmatrix} x&y&z\\z&x^2&y^2\end{pmatrix} }      \\
W_8   & (4,5,6)    &  
y^2-xz,\; z^2-x^3      \\ [1mm]
U_9    &   (3,5,7)+(\n-,\n-,1)  & y^2-xz,\; yz-x^4\\ 
 U_8   &(2,3,\n-)+(1,\n-,2)+(\n-,\n-,1) &zy,\; y^2-x^3+zx \\
 U_7^*  & (3,4,5)+(\n-,1,\n-)   &  \scalefont{0.80}{\arraycolsep=3pt
\begin{pmatrix} x&y&z\\z&x^2&xy\end{pmatrix} }   \\  
 U_7   &  (3,4,5)+(\n-,\n-,1)   & y^2-xz,\; yz-x^3\\  [1mm]
T_9    &   (2,3,\n-)+(\n-,5,2)   & xz,\; y^2-z^5-x^3    \\
T_8    &   (2,3,\n-)+(\n-,\n-,1)+(\n-,2,1) & xz,\; y^2-yz^2-x^3   \\
T_7^*    &   (2,3,\n-)+(\n-,2,3)   &  \scalefont{0.80}{\arraycolsep=3pt
\begin{pmatrix} x&y&z\\0&z&y^2-x^3\end{pmatrix} } \\
T_7    &   (2,3,\n-)+(\n-,3,2)    & xz,\; y^2-z^3-x^3    \\ [1mm]
 E_{12}(2)    &  (3,7,8)       &  \scalefont{0.80}{\arraycolsep=3pt
\begin{pmatrix} x^2&y&z\\y&z&x^3\end{pmatrix} }  \\
 J_{2,1}(2)  &  (2,5,\n-)+(1,\n-,2)    &  \scalefont{0.80}{\arraycolsep=3pt
\begin{pmatrix} z&y&x^3\\0&x^2-z&y\end{pmatrix} } \\
 J_{2,0}(2)  & (1,\n-,\n-)+(1,2,\n-)+(1,\n-,2)   &  \scalefont{0.80}{\arraycolsep=3pt
\begin{pmatrix} z&y-x^2&0\\0&x^2-z&y\end{pmatrix} }\\ [1mm]
 E_8(1) &(3,5,7)   &  \scalefont{0.80}{\arraycolsep=3pt
\begin{pmatrix} x&y&z\\y&z&x^3\end{pmatrix} } \\ 
 E_7(1) &(2,3,\n-)+(1,\n-,2)   &  \scalefont{0.80}{\arraycolsep=3pt
\begin{pmatrix} z&x&y\\0&y&x^2-z\end{pmatrix} } \\
 E_6(1) &(3,4,5)    &  \scalefont{0.80}{\arraycolsep=3pt
\begin{pmatrix} x&y&z\\y&z&x^2\end{pmatrix} } \\ [1mm]
S_{2k+3}  &  (1,\n-,\n-)+(1,k,\n-)+(\n-,\n-,1)+(\n-,1,1) 
          &xz,\; y^2-yx^{k}-yz  \\
S_{2k+4}  &  (2,2k+1,\n-)+(\n-,\n-,1)+(\n-,1,1) 
             &xz,\; y^2-x^{2k+1}-yz \\
S_{6}^*  &  (2,3,\n-)+(\n-,\n-,1)+(1,\n-,1) &  \scalefont{0.80}{\arraycolsep=3pt
\begin{pmatrix} z&x&y\\0&y&x^2-xz\end{pmatrix} } \\
\bottomrule
\end{array}
$
\end{table}

\end{document}